\documentclass[12pt]{amsart}
\usepackage[T1]{fontenc}
\usepackage{lmodern}
\usepackage[a4paper]{geometry}
\usepackage{amsmath,amssymb,amsthm,mathtools}
\usepackage{booktabs,tabularx,array}
\usepackage[backref,colorlinks=true,linkcolor=blue,citecolor=blue,urlcolor=blue]{hyperref}
\usepackage{bm}

\theoremstyle{plain}
\newtheorem{thm}{Theorem}[section]
\newtheorem{coro}[thm]{Corollary}
\newtheorem{lem}[thm]{Lemma}
\newtheorem{prop}[thm]{Proposition}

\theoremstyle{definition}
\newtheorem{conj}[thm]{Conjecture}
\newtheorem{rmk}[thm]{Remark}

\numberwithin{equation}{section}
\numberwithin{figure}{section}

\newcommand{\bflam}{\bm{\lambda}}
\newcommand{\dd}{\mathbf d}
\newcommand{\im}{\operatorname{im}}
\newcommand{\Ldown}{L^{\operatorname{down}}}
\newcommand{\Lup}{L^{\operatorname{up}}}
\newcommand{\lk}{\operatorname{lk}}
\newcommand{\p}{\partial}
\newcommand{\pre}{\preccurlyeq}
\newcommand{\R}{\mathbb R}
\newcommand{\rank}{\operatorname{rank}}
\newcommand{\tr}{\operatorname{tr}}

\title[Laplacian Eigenvalue Sums of Complexes]{Degree Majorization and Laplacian Eigenvalue Sums for Simplicial Complexes}

\author [H.-Z. Zhang]{Huan-Zhi Zhang}
\address{\small School of Mathematical Sciences, Anhui University, Hefei 230601, P. R. China}
\email{zhanghz@stu.ahu.edu.cn}

\author[Y.-M. Song]{Yi-Min Song$^\dag$}
\address{School of Mathematics and Physics, Anhui Jianzhu University, Hefei 230601, P. R. China}
\email{songym@ahjzu.edu.cn}
\thanks{$^\dag$Supported by National Natural Science Foundation of China (Grant No. 12501468).}

\author [Y.-Z. Fan]{Yi-Zheng Fan*}
\address{\small Center for Pure Mathematics, School of Mathematical Sciences, Anhui University, Hefei 230601, P. R. China}
\email{fanyz@ahu.edu.cn}
\thanks{*Corresponding author. Supported by National Natural Science Foundation of China (No. 12331012).}
	
	

\begin{document}

\begin{abstract}
Let $K$ be an $r$-dimensional simplicial complex.
We prove that the spectrum of its $(r - 1)$-dimensional up-Laplacian is majorized by the conjugate degree sequence of its $(r - 1)$-dimensional faces:
\[
\bflam_{r-1}(K) \pre \dd_{r-1}^\top(K).
\]
We also establish a Brouwer-type inequality: for every integer $\ell \geq 1$,
\[
\sum_{i = 1}^{\ell}\lambda_{r-1,i}(K)
\leq
\frac{r + 1}{2}f_r(K)
+
\frac{f_{r - 2}(K)}{r}
\binom{\ell + 1}{2},
\]
where $\lambda_{r-1,i}(K)$ denotes the $i$-th largest eigenvalue in the spectrum $\bflam_{r-1}(K)$, and $f_t(K)$ denotes the number of $t$-dimensional faces of $K$.
These results provide higher-dimensional analogs of the
Grone-Merris-Bai theorem and the Brouwer-Kothari-Tudose theorem 
and recover the corresponding graph results when $r=1$.

We show that the Duval-Reiner conjecture on the majorization by the conjugate degree sequence of vertices fails in every dimension $r \geq 2$.
More precisely, for every $n \geq r + 5$, we construct a pure $r$-dimensional complex on $n$ vertices that violates the conjectured inequality at the fifth partial sum.
\end{abstract}

\maketitle

\section{Introduction}
An \emph{abstract simplicial complex}, or simply a \emph{complex} $K$ on a finite set $V$ is a collection of subsets of $V=V(K)$ closed under inclusion. 
Throughout the paper, all complexes or graphs are finite.
An element of $K$ of cardinality $i+1$ is called an \emph{$i$-face} or an \emph{$i$-simplex} of $K$, and its dimension is $i$.
Usually a $0$-face is called a \emph{vertex}, and a $1$-face is called an \emph{edge}.
The \emph{dimension} of $K$ is the maximum dimension of all faces.
Let $C_i(K;\R)$ be the real \emph{$i$-chain space} of $K$ spanned by oriented $i$-faces of $K$, and let $\p_i: C_i(K;\R)\to C_{i-1}(K;\R)$ be the \emph{$i$-th boundary map} of $K$.
It is known that $\p_{i+1}\p_i =0$, and the \emph{$i$-th homology group} of $K$ is defined by $H_i(K;\R):=\ker \p_i / \im \p_{i+1}$.
With respect to the standard inner products on the chain spaces, let $\p_i^*$ be the adjoint operator of $\p_i$.
The \emph{$i$-dimensional Laplacian} of $K$ is defined by
\[
L_i(K) = \p_{i + 1}\p_{i + 1}^{*} + \p_i^{*}\p_i,
\]
where $\Lup_i(K):=\p_{i + 1}\p_{i + 1}^{*}$ is the \emph{$i$-dimensional up Laplacian} of $K$, and $\Ldown_i(K):=\p_i^{*}\p_i$ is the \emph{$i$-dimensional down Laplacian} of $K$.
In 1944, Eckmann \cite{E1944} formulated and proved the discrete version of the Hodge theorem:
\[
\widetilde H_i(K;\R) \cong \ker(L_i(K)).
\]

The Laplacian spectra of simplicial complexes have received a lot of attention.
Horak and Jost \cite{HJ2013a, HJ2013b} developed a general framework for weighted Laplace operators on simplicial complexes and established general eigenvalue bounds and interlacing inequalities.
Duval and Reiner \cite{DR2002} studied Laplacian spectra of shifted simplicial complexes, while Duval, Klivans, and Martin \cite{DKM2009} proved the simplicial matrix-tree theorems.
Aharoni, Berger, and Meshulam \cite{ABM2005} established inequalities relating the smallest eigenvalues of successive higher Laplacians of flag complexes and derived a spectral criterion for the vanishing of their reduced homology, and Lew \cite{L2020b} extended these inequalities to general complexes with bounded missing face dimension.
Lew \cite{L2020a} also obtained the lower bound for the smallest eigenvalue of the Laplacian for complexes with bounded missing face dimension and further conjectured the corresponding extremal construction.
Zhan, Huang, and Lin \cite{ZHL2026} subsequently proved this conjecture.

Recent work has also studied the bound for the largest up-Laplacian eigenvalue of complexes.
Fan, Wu, and Wang \cite{FWW2025} proved that the largest up-Laplacian eigenvalue is at most the spectral radius of the corresponding signless up-Laplacian, and characterized equality under $r$-path connectedness.
Recently, Zhang and Fan \cite{ZF2026b} established a sharper upper bound for the largest up-Laplacian eigenvalue in terms of unions of neighborhoods of the $(r - 1)$-faces contained in an $r$-face.
Signless Laplacians on simplicial complexes were introduced in \cite{KO2020} and have since been studied in several directions; see, for example, \cite{FSZ2025,FZ2026,L2014,SFS2026,ZF2026a}.
Normalized Laplacians and incidence balancedness were studied in \cite{SWF2025}.

In this paper, we are focusing on the bound of the up-Laplacian eigenvalue sum in terms of the degrees of faces.
Let $K$ be an $r$-dimensional complex on vertex set $V$ with $r \ge 1$.
For $-1 \leq i \leq r$, let $S_i(K)$ be the set of $i$-faces and write $f_i(K) = |S_i(K)|$, with $S_{-1}(K) = \{\emptyset\}$ and $f_{-1}(K) = 1$.
Let $\bflam_i(K):= \bigl(\lambda_{i,1} \ge \lambda_{i,2} \ge \cdots \ge \lambda_{i,f_i(K)}\bigr)$ denote the spectrum of $\Lup_{i}(K)$ arranged in nonincreasing order.
For every $i$-face $\sigma$ of $K$, the \emph{degree} of $\sigma$ is defined by
$$ \deg_i(K; \sigma)=|\{F \in S_r(K): \sigma \subseteq F\}|.$$
The \emph{$i$-degree sequence} of $K$, denoted by $\dd_i(K)=(\deg_{i,1}(K) \ge \cdots \ge \deg_{i,f_i(K)}(K))$, is the sequence of degrees of all $i$-faces of $K$ arranged in nonincreasing order.
The	\emph{conjugate $i$-degree sequence} of $K$, denoted by $$\dd_i^\top(K)=(\deg^\top_{i,1}(K) \ge \deg^\top_{i,2}(K) \ge \cdots),$$ is the conjugate of $\dd_i(K)$, defined by
\[
\deg^\top_{i,j}(K):=\bigl|\{t:\deg_{i,t}(K) \geq j\}\bigr|, ~ j \ge 1.
\]

All sequences are padded with zeros if we consider the majorization between two sequences below.
For nonincreasing finitely supported sequences $\mathbf{x}$ and $\mathbf{y}$, we say that \emph{$\mathbf{x}$ is majorized by $\mathbf{y}$}, denoted by $\mathbf{x} \pre \mathbf{y}$, if
\[
\sum_{i = 1}^{\ell} x_i \leq \sum_{i = 1}^{\ell} y_i
~ \text{for every $\ell \geq 1$},
 \text{and~}
\sum_{i \geq 1} x_i = \sum_{i \geq 1} y_i.
\]

Viewing a graph as a one-dimensional simplicial complex gives the following theorem.

\begin{thm}[Grone-Merris-Bai \cite{B2011, GM1994}]\label{gmb}
For every simple graph $G$,
\[
\bflam_{0}(G) \pre \dd_0^\top(G).
\]
\end{thm}

Duval and Reiner proposed a higher-dimensional majorization by the conjugate vertex-degree sequence.

\begin{conj}[Duval-Reiner \cite{DR2002}]\label{dr}
Let $K$ be an $r$-dimensional simplicial complex.
Then
\begin{equation}\label{DR-eq}
\bflam_{r-1}(K) \pre \dd_0^\top(K).
\end{equation}
\end{conj}

Very recently, Han, Lu, and Wang \cite{HLW2026} proved the first two partial-sum inequalities in a related formulation for the curl-curl operator on graphs and, more generally, for up-Laplacians of $3$-families.
In Section \ref{counter}, we show that Conjecture \ref{dr} generally fails by constructing pure complexes that violate its fifth partial-sum inequality in \eqref{DR-eq}.
While we prepare the paper, Huang \cite{Huang} also found countexamples to Duval-Reiner conjecture.
They confirmed the second partial-sum inequality in \eqref{DR-eq} and characterized the corresponding equality case.

Our first main result replaces the vertex degree sequence by the $(r-1)$-degree sequence in Conjecture \ref{dr}.
When $r=1$, Theorem \ref{main-gm} is exactly the Grone-Merris-Bai theorem.

\begin{thm}\label{main-gm}
Let $K$ be an $r$-dimensional simplicial complex.
Then
\[
\bflam_{r-1}(K) \pre \dd_{r-1}^\top(K).
\]
\end{thm}

Kothari and Tudose recently proved Brouwer's Laplacian conjecture for graphs.

\begin{thm}[Brouwer-Kothari-Tudose\cite{BH2012, KT2026}]
\label{brou}
Let $G$ be a simple graph.
Then, for every integer $\ell\geq 1$,
\[
\sum_{i = 1}^{\ell}\lambda_{0,i}(G)
\leq
f_1(G) + \binom{\ell + 1}{2}.
\]
\end{thm}

In Theorem \ref{brou}, $f_1(G)$ is exactly the number of edges of $G$. Lew proposed the following higher-dimensional analog.

\begin{conj}[Lew \cite{L2025}]\label{lew}
Let $K$ be an $r$-dimensional simplicial complex.
Then, for every integer $\ell \geq 1$,
\[
 \sum_{i = 1}^{\ell}\lambda_{r-1,i}(K)
\leq f_r(K) + \binom{\ell}{2} + r\ell.
\]
\end{conj}

Our second main result gives the following weaker bound, which recovers Brouwer-Kothari-Tudose theorem when $r=1$.

\begin{thm}\label{main-b}
Let $K$ be an $r$-dimensional simplicial complex.
Then for every integer $\ell \geq 1$,
\[
\sum_{i = 1}^{\ell}\lambda_{r-1,i}(K)
\leq
\frac{r+1}{2}\,f_r(K)
+
\frac{f_{r-2}(K)}{r}
\binom{\ell+1}{2}.
\]
\end{thm}

To compare the two right-hand sides, note that $f_r(K)\geq 1$ and $f_{r-2}(K)\geq\binom{r+1}{2}$.
Consequently, the right-hand side in Theorem~\ref{main-b} exceeds the conjectured right-hand side of Conjecture \ref{lew} by at least
\[
\frac{(r-1)(\ell-1)(\ell-2)}{4}\geq 0.
\]

In particular, taking $\ell=1$, we get an upper bound for the largest $(r-1)$-dimensional up Laplacian eigenvalue.

\begin{coro}
Let $K$ be an $r$-dimensional simplicial complex.
Then,
\[
\lambda_{r-1,1}(K)
\leq
\frac{r+1}{2}\,f_r(K) + \frac{f_{r-2}(K)}{r}.
\]
\end{coro}

The proofs of both main theorems rely on a relation between 
the $r$-dimension down Laplacian of $K$ and the Laplacian of the link graphs with respect to $(r-2)$-faces.
Indeed, we prove that
$$ r\Ldown_r(K)=\sum_{\eta \in S_{r-2}(K)} \Ldown_r(K;\eta),$$
where $ \Ldown_r(K;\eta)$ is a local $r$-dimensional down Laplacian defined on the $r$-faces of $K$ containing an $(r-2)$-face $\eta$.
Note that for the $(r-2)$-face $\eta$, its link $\lk \eta$ is a graph.
We show that, up to a signed permutation,
 $$ \Ldown_r(K;\eta)=\Ldown_1(\lk \eta) \oplus O.$$
As  $\Lup_0(\lk \eta)$ shares the same nonzero eigenvalues with $\Ldown_1(\lk \eta)$, we can apply Theorems \ref{gmb} and \ref{brou} to the graph $\lk \eta$, and get the desired results.

The paper is organized as follows.
Section \ref{pre} introduces the Laplacian of simplicial complexes and some linear algebra results on the partial sum of eigenvalues.
Section \ref{graph-decom} establishes a decomposition of the down Laplacian in terms of local down Laplacians.
Section \ref{main-proof} proves Theorems \ref{main-gm} and \ref{main-b}.
Section \ref{counter} constructs counterexamples to Conjecture \ref{dr}.

\section{Preliminaries}\label{pre}

\subsection{Laplacian of simplicial complexes}
Let $K$ be a finite abstract simplicial complex with a linearly ordered vertex set $V=V(K)$.
We orient each face by the induced increasing order on its vertices.
For a face $\sigma \in K$, write $[\sigma]$ the oriented face of $\sigma$.
For $i\geq 0$, the real $i$-chain space $C_i(K;\R)$ is the vector space over $\R$ spanned by all oriented $i$-faces of $K$, and we set $C_{-1}(K;\R)=\R$.
The $i$-th boundary map $\p_i:C_i(K;\R)\to C_{i-1}(K;\R)$ is defined by
\[
\p_i[v_0,\ldots,v_i]
=
\sum_{j=0}^i(-1)^j[v_0,\ldots,\widehat v_j,\ldots,v_i],
\]
where $\widehat v_j$ indicates that $v_j$ is omitted.
For $\tau \in S_{i - 1}(K)$ and $\sigma \in S_i(K)$, let $([\sigma]:[\tau])$  be the sign of $[\tau]$ appeared in $\p_i [\sigma]$, and set $([\sigma]: [\tau]) = 0$ if $\tau \not\subset \sigma$.
Thus map $\p_i$ is represented by the \emph{boundary matrix}
$B_i(K) \in \R^{f_{i-1}(K) \times f_i(K)}$,
whose entries are given by
\begin{equation}\label{BouMat}
(B_i(K))_{\tau,\sigma}
=
([\sigma]:[\tau]).
\end{equation}
The identities $\p_i\p_{i+1}=0$ give the augmented chain complex
\[
\cdots \xrightarrow{\p_{i+1}} C_i(K;\R) \xrightarrow{\p_i} C_{i-1}(K;\R)
\to\cdots \xrightarrow{\p_1} C_0(K;\R) \xrightarrow{\p_0} C_{-1}(K;\R) \to 0.
\]

For each $i \geq 0$, equip $C_i(K;\R)$ with a positive definite inner product that makes the basis consisting of oriented $i$-faces orthonormal.
The adjoint operator $\p_i^*: C_{i-1}(K;\R) \to C_{i}(K;\R)$ of $\p_i$ is defined by
\[
\langle \p_i g_1, g_2 \rangle_{C_{i-1}(K;\R)} = \langle g_1, \p_i^* g_2\rangle_{C_i(K;\R)}
\]
for all $g_1 \in C_i(K;\R)$, $g_2 \in C_{i-1}(K;\R)$.
The \emph{$i$-dimensional up Laplacian} and \emph{$i$-dimensional down Laplacian} are, respectively, defined by
\[
\Lup_{i}(K) :=\p_{i+1}\p_{i+1}^* = B_{i+1}(K) B_{i+1}(K)^\top,
\quad
\Ldown_{i}(K) :=\p_{i}^*\p_{i} = B_i(K)^\top B_i(K).
\]

When $K$ is a one-dimensional complex, that is, a graph, then $\Lup_{0}(K)$ is exactly the usual Laplacian of a graph,
and $\Ldown_1(K)$ is the edge Laplacian of $K$.

\subsection{Partial sum of eigenvalues}
For a positive semidefinite matrix $A\in\R^{n\times n}$, let
\[
\mu_1(A)\geq\cdots\geq\mu_n(A)\geq 0
\]
be its eigenvalues, and set $\mu_i(A)=0$ for $i>n$.
For every integer $\ell\geq 1$, define
\[
\Phi_\ell(A):=\sum_{i=1}^{\ell}\mu_i(A).
\]

\begin{lem}[Ky Fan \cite{F1949}]\label{ky}
Let $A\in\R^{n\times n}$ be a symmetric matrix with eigenvalues
$\mu_1(A)\geq\mu_2(A)\geq\cdots\geq\mu_n(A)$.
Then, for every integer $1\leq\ell\leq n$,
\[
\sum_{i=1}^{\ell}\mu_i(A)
=
\max_{\substack{U\in\R^{n\times\ell} \\ U^T U=I_\ell}} \tr(U^\top A U)
=
\max_{\substack{P=P^\top=P^2 \\ \rank(P)=\ell}} \tr(P A).
\]
\end{lem}

\begin{lem}\label{ky-sum}
Let $A_1,\ldots,A_t$ be positive semidefinite matrices of the same order $n$.
Then, for every integer $\ell \geq 1$,
\begin{equation}\label{eq-ky-sum}
\Phi_\ell\left(\sum_{j = 1}^t A_j\right)
\leq \sum_{j = 1}^t \Phi_\ell(A_j).
\end{equation}
\end{lem}

\begin{proof}
Suppose that $1 \leq \ell < n$.
By Lemma \ref{ky}, for every real symmetric matrix $A$ of size $n \times n$,
\begin{equation}\label{eq-ky}
\Phi_\ell(A)
= \max\{\tr(PA) : P = P^\top = P^2,\ \rank(P) = \ell\}.
\end{equation}
Applying \eqref{eq-ky} to $\sum_{j = 1}^t A_j$ gives
\[
\Phi_\ell\left(\sum_{j = 1}^t A_j\right)
= \max_P \sum_{j = 1}^t \tr(PA_j)
\leq \sum_{j = 1}^t \max_P \tr(PA_j)
= \sum_{j = 1}^t \Phi_\ell(A_j).
\]

Now assume that $\ell \geq n$.
Since each $A_j$ is positive semidefinite and its spectrum is padded with zeros, $\Phi_\ell(A_j) = \tr(A_j)$.
The same identity holds for $\sum_{j = 1}^t A_j$.
Therefore, the additivity of the trace shows that \eqref{eq-ky-sum} holds with equality in this case.
\end{proof}

\section{Decomposition of Laplacian}\label{graph-decom}
Let $K$ be an $r$-dimensional complex, and let $\sigma$ be a face of $K$.
The \emph{link} of $\sigma$, denoted by $\lk \sigma$, is defined as
$$ \lk \sigma=\{\tau \in K: \tau \cap \sigma = \emptyset, \tau \cup \sigma \in K\}.$$
In particular, if $\eta \in S_{r - 2}(K)$, then
$\lk \eta$ is a graph.
So, we have the $1$-dimensional down Laplaican $\Ldown_1(\lk \eta)$, i.e. the edge-Laplacian of the graph $\lk \eta$.
Indeed, for every $\eta \in S_{r - 2}(K)$, the boundary matrix $B_1(\lk \eta) \in \R^{f_0(\lk \eta) \times f_1(\lk \eta)}$ is given by
\[
(B_1(\lk \eta))_{v,e}=
\begin{cases}
([e],[v]), & \text{if } v \in e,\\
0, & \text{otherwise},
\end{cases}
\]
where $v \in S_0(\lk \eta)$ and $e \in S_1(\lk \eta)$.
Thus,
$$ \Ldown_1(\lk \eta)=B_1(\lk \eta)^\top B_1(\lk \eta).$$

We now extend $B_1(\lk \eta) \in \R^{f_0(\lk \eta) \times f_1(\lk \eta)}$ to a matrix $B_k(K;\eta) \in \R^{f_{r - 1}(K) \times f_r(K)}$, which is defined  by
\[
(B_k(K;\eta))_{\tau,\sigma}
=
\begin{cases}
([\sigma] : [\tau]), & \text{if } \eta \subset \tau \subset \sigma,\\
0, & \text{otherwise},
\end{cases}
\]
where $\tau \in S_{r-1}(K)$ and $\sigma \in S_r(K)$.
The matrix 
$$\Ldown_r(K;\eta):= B_k(K;\eta)^\top B_k(K;\eta)$$
 is a local $r$-dimensional down Laplaican defined on the $r$-faces of $K$ containing the face $\eta$.

\begin{lem}\label{gram-entries}
Let $K$ be an $r$-dimensional complex.
Then, for any $\sigma_1,\sigma_2 \in S_r(K)$, 
\[
(\Ldown_r(K;\eta))_{\sigma_1,\sigma_2}
=
\begin{cases}
2, & \text{if } \sigma_1 = \sigma_2 \text{ and } \eta \subset \sigma_1, \\
([\sigma_1] : [\tau])([\sigma_2] : [\tau]),
& \text{if } \sigma_1 \ne \sigma_2,\ \tau = \sigma_1 \cap \sigma_2 \in S_{r - 1}(K), \eta \subset \tau, \\
0, & \text{otherwise}.
\end{cases}
\]
Moreover, after reordering the vertices and relabeling $r$-faces of $K$, 
$$ \Ldown_r(K;\eta)= \Ldown_1(\lk \eta) \oplus O.$$
\end{lem}

\begin{proof}
By the definition of $D_\eta$, we have
\[
(\Ldown_r(K;\eta))_{\sigma_1,\sigma_2}
= \sum_{\tau \in S_{r - 1}(K)}
(B_k(K;\eta))_{\tau,\sigma_1}(B_k(K;\eta))_{\tau,\sigma_2}
= \sum_{\substack{\tau \in S_{r - 1}(K)\\
\eta \subset \tau \subset \sigma_1 \cap \sigma_2}}
([\sigma_1] : [\tau])([\sigma_2] : [\tau]).
\]

First, suppose that  $\sigma_1 = \sigma_2 = \sigma$.
If $\eta \not\subset \sigma$, then the sum is empty.
If $\eta \subset \sigma$, then exactly two $(r - 1)$-dimensional faces $\tau$ satisfy $\eta \subset \tau \subset \sigma$.
Each of these faces contributes $([\sigma] : [\tau])^2 = 1$, so the diagonal entry is $2$.

Now suppose that $\sigma_1 \ne \sigma_2$.
Any $\tau$ occurring in the sum is a common $(r - 1)$-dimensional face of $\sigma_1$ and $\sigma_2$.
 This face exists if and only if $\tau = \sigma_1 \cap \sigma_2 \in S_{r - 1}(K)$ and $\eta \subset \tau$.
In this case, the sum consists of the single term $([\sigma_1] : [\tau])([\sigma_2] : [\tau])$.
In all other cases, the sum is empty and the corresponding entry is $0$.

Assume that $\eta \subset \sigma$, and write $e=\sigma \setminus \in S_1(\lk \eta)$.
It is easy to verify
$$ (\Ldown_1(\lk \eta))_{e,e}=(\Ldown_r(K;\eta))_{\sigma, \sigma}=2.$$
Suppose that $\eta  \subset \sigma_1$, $\eta \subset \sigma_2$, $\sigma_1 \ne \sigma_2$, and $\tau=\sigma_1 \cap \sigma_2 \in S_{r-1}(K)$.
Write $e_1=\sigma_1 \setminus \eta, e_2 =\sigma_2 \setminus \eta$, and $\{v\}=e_1 \cap e_2$.
In addition, assume that each vertex of $\eta$ is smaller than any vertex of $e_1 \cup e_2$ with respect to the liner order of the vertices of $K$. 
Thus, by the above discussion, 
\[
\begin{split}
(\Ldown_r(K;\eta))_{\sigma_1,\sigma_2} & =
([\sigma_1] : [\tau])([\sigma_2] : [\tau])\\ 
&=([\eta \cup e_1]: [\eta \cup \{v\}])([\eta \cup e_2]: [\eta \cup \{v\}]\\
&= (-1)^{r-1} ([e_1]: [v]) \cdot (-1)^{r-1}([e_2]: [v])\\
& = ([e_1]: [v]) ([e_2]: [v])\\
&= (\Ldown_1(\lk \eta))_{e_1,e_2}.
\end{split}
\]
So, after relabeling $r$-faces of $K$, we have 
$$ \Ldown_r(K;\eta)= \Ldown_1(\lk \eta) \oplus O.$$

In general, if the orders of the vertices of $\eta$ do not satisfy the above requirement, by \cite[Lemma 3.2]{FSZ2025}, the Laplacian $\Ldown_r(K;\eta)$ in this case is similar to $\Ldown_1(\lk \eta) \oplus O$ via a signature matrix (a diagonal matrix with $\pm 1$ on the diagonal).
\end{proof}

\begin{lem}\label{up-decom}
Let $K$ be an $r$-dimensional complex. 
Then
\begin{equation}\label{eq-up-decom}
r \Ldown_r(K) = \sum_{\eta \in S_{r - 2}(K)} \Ldown_r(K; \eta).
\end{equation}
\end{lem}

\begin{proof}
It suffices to compare the entries of the two sides of \eqref{eq-up-decom}.
Let $B_r(K)$ be the boundary matrix defined as in \eqref{BouMat}.
Then $\Ldown_r(K)=B_r(K)^\top B_r(K)$.
First fix $\sigma \in S_r(K)$.
The face $\sigma$ has exactly $r + 1$ faces of dimension $r - 1$, hence 
$$(\Ldown_r(K))_{\sigma,\sigma}=(B_r(K)^\top B_r(K))_{\sigma,\sigma} = r + 1.$$
So, the $(\sigma,\sigma)$-entry on the left-hand side of \eqref{eq-up-decom} is $r(r + 1)$.
The face $\sigma$ contains exactly $\binom{r + 1}{2}$ faces $\eta$ of dimension $r - 2$.
For each such face $\eta$, Lemma \ref{gram-entries} gives $(\Ldown_r(K; \eta))_{\sigma,\sigma} = 2$.
Thus, the corresponding entry on the right-hand side is $2\binom{r + 1}{2} = r(r + 1)$.

Now, let $\sigma_1,\sigma_2 \in S_r(K)$ be distinct.
If $\sigma_1$ and $\sigma_2$ do not share an $(r - 1)$-dimensional face, then $(B_r(K)^\top B_r(K))_{\sigma_1,\sigma_2} = 0$.
In this case, Lemma \ref{gram-entries} gives $(\Ldown_r(K; \eta))_{\sigma_1,\sigma_2} = 0$ for every $\eta \in S_{r - 2}(K)$.
So we assume that $\tau = \sigma_1 \cap \sigma_2 \in S_{r - 1}(K)$.
Then $(B_r^\top B_r)_{\sigma_1,\sigma_2} = ([\sigma_1] : [\tau])([\sigma_2] : [\tau])$.
The face $\tau$ contains exactly $r$ faces $\eta$ of dimension $r - 2$.
For each such face $\eta$, Lemma \ref{gram-entries} gives $(\Ldown_r(K; \eta))_{\sigma_1,\sigma_2} = ([\sigma_1] : [\tau])([\sigma_2] : [\tau])$, while every other $\eta$ contributes zero.
Therefore,
\[
\sum_{\eta \in S_{r - 2}(K)}
(\Ldown_r(K; \eta))_{\sigma_1,\sigma_2}
=
r([\sigma_1] : [\tau])([\sigma_2] : [\tau])
=
(rB_r^\top B_r)_{\sigma_1,\sigma_2}=(r\Ldown_r(K))_{\sigma_1,\sigma_2}.
\]
\end{proof}

\begin{lem}\label{com-graph}
Let $K$ be an $r$-dimensional complex.
Then for every integer $\ell \geq 1$,
\[
\sum_{i = 1}^\ell \lambda_{r-1,i}(K)
\leq
\frac{1}{r}
\sum_{\eta \in S_{r - 2}(K)}
\Phi_\ell(\Ldown_r(K;\eta))
=
\frac{1}{r}\sum_{\eta \in S_{r - 2}(K)}\Phi_\ell(\Lup_0(\lk \eta)).
\]
\end{lem}

\begin{proof}
Since $\Lup_{r - 1}(K) = B_r B_r^\top$ shares the same nonzero eigenvalues with $\Ldown_r(K)=B_r^\top B_r$ (including multiplicities), we have
\[
\Phi_\ell\left(\Lup_{r - 1}(K)\right)
=
\Phi_\ell(\Ldown_r(K)).
\]
Similarly, $\Phi_\ell(\Ldown_1(\lk \eta))=\Phi_\ell(\Lup_0(\lk \eta))$.
Then, by Lemmas \ref{ky-sum}, \ref{gram-entries}, and \ref{up-decom}, we obtain
\[
\begin{aligned}
\sum_{i = 1}^\ell \lambda_{r-1,i}(K)
&= \Phi_\ell\left(\Lup_{r - 1}(K)\right) = \Phi_\ell(\Ldown_r(K))
= \frac{1}{r}
\Phi_\ell\left(
\sum_{\eta \in S_{r - 2}(K)}
\Ldown_r(K;\eta)
\right) \\
&\leq \frac{1}{r}
\sum_{\eta \in S_{r - 2}(K)}
\Phi_\ell(\Ldown_r(K;\eta))
=
\frac{1}{r}\sum_{\eta \in S_{r - 2}(K)}\Phi_\ell(\Ldown_1(\lk \eta))\\
&=\frac{1}{r}\sum_{\eta \in S_{r - 2}(K)}\Phi_\ell(\Lup_0(\lk \eta)).
\end{aligned}
\]
\end{proof}

\section{Proofs of main theorems}\label{main-proof}
\begin{proof}[Proof of Theorem \ref{main-gm}]
Let $K$ be an $r$-dimensional complex.
For every integer $\ell\geq 1$, double counting the pairs $(\tau,j)$ satisfying $\tau\in S_{r-1}(K)$ and $1\leq j\leq\min\{\deg_{r-1}(K;\tau),\ell\}$ gives
\[
\sum_{j = 1}^{\ell} \deg^\top_{r - 1,j}(K)
=
\sum_{\tau \in S_{r - 1}(K)}
\min\{\deg_{r-1}(K;\tau),\ell\}.
\]
Fix $\eta \in S_{r-2}(K)$.
For $v \in V(\lk \eta)$, the edges of $\lk \eta$ incident with $v$ are in bijection with the $r$-faces of $K$ containing $\{v\} \cup \eta$.
Therefore $\deg_{1}(\lk \eta; v)=\deg_{r-1}(K;\{v\} \cup \eta)$.
Applying Theorem \ref{gmb} to $\lk \eta$ and double counting the pairs $(v,j)$ with $v \in V(\lk,\eta)$ and $1\leq j\leq\min\{\deg_{1}(\lk \eta;v),\ell\}$ gives
\[
\begin{split}
\Phi_\ell(\Lup_0(\lk \eta))
 & \leq
\sum_{j = 1}^{\ell}\deg^\top_{0,j}(\lk \eta)
=\sum_{v \in V(\lk \eta)} \min\{\deg_{1}(\lk \eta;v),\ell\}\\
&=\sum_{\substack{\eta \subset \tau \in S_{r - 1}(K) }}
\min\{\deg_{r - 1}(K;\tau),\ell\}.
\end{split}
\]

Suppose first that $1\leq\ell\leq f_r(K)$.
By Lemma \ref{com-graph}, we have
\[
\begin{aligned}
\sum_{i = 1}^{\ell} \lambda_{r-1,i}(K)
&\leq \frac{1}{r}
\sum_{\eta \in S_{r - 2}(K)}
\Phi_\ell(\Lup_0(\lk \eta)) \\
&\leq
\frac{1}{r}
\sum_{\eta \in S_{r - 2}(K)}
\sum_{\substack{\eta \subset \tau \in S_{r - 1}(K)}}
\min\{\deg_{r - 1}(K;\tau),\ell\} \\
&=
\sum_{\tau \in S_{r - 1}(K)}
\min\{\deg_{r - 1}(K;\tau),\ell\}
=
\sum_{j = 1}^{\ell} \deg^\top_{r - 1,j}(K).
\end{aligned}
\]
If $\ell>f_r(K)$, as $\rank\Lup_{r-1}(K)\leq f_r(K)$, we have  $\lambda_{r-1,j}(K)=0$ for $j>f_r(K)$.
Also, $\deg_{r-1}(K;\tau)\leq f_r(K)$ for every $\tau$, so $\deg^\top_{r-1,j}(K)=0$ for $j>f_r(K)$.
Thus the partial-sum inequality for $\ell$ reduces to the already proved case $\ell \le f_r(K)$.

We now check equality of the total sums.
By definition, $\tr (\Ldown_r(K))= (r + 1)f_r(K)$.
Double counting the pairs $(\tau,\sigma)$ satisfying $\tau \in S_{r - 1}(K)$, $\sigma \in S_r(K)$, and $\tau \subset \sigma$ gives
\[
\sum_{\tau \in S_{r - 1}(K)} \deg_{r-1}(K;\tau) = (r + 1)f_r(K).
\]
Since $\deg_{r-1}(K;\tau) \leq f_r(K)$ for every $\tau \in S_{r - 1}(K)$, 
$$\sum_{j \ge 1} \deg^\top_{r - 1,j}(K) = \sum_{j = 1}^{f_r(K)} \deg^\top_{r - 1,j}(K)
=
\sum_{\tau \in S_{r - 1}(K)}\deg_{r-1}(K;\tau)=(r + 1)f_r(K).$$
Consequently,
\[
\sum_{i \geq 1} \lambda_{r-1,j}(K)
= \tr\left(\Lup_{r - 1}(K)\right)
= \tr(\Ldown_r(K))
= (r + 1)f_r(K) = \sum_{j \geq 1} \deg^\top_{r - 1,j}(K).
\]
By the above discussion, we get $\bflam_{r - 1}(K) \pre \dd_{r - 1}^\top(K)$.
\end{proof}

\begin{rmk}
Let $\Delta_{r+1}$ be the $r$-dimensional simplex on $r+1$ vertices.
Then
\[
\bflam_{r-1}(\Delta_{r+1})
=
(r+1,0,\ldots,0)
=
\dd_{r-1}^\top(\Delta_{r+1}),
\]
so equality holds in Theorem \ref{main-gm}.
\end{rmk}

\begin{proof}[Proof of Theorem \ref{main-b}]
Fix an integer $\ell \geq 1$.
For every $\eta \in S_{r - 2}(K)$, by Lemma \ref{gram-entries}, we have 
$$\Phi_\ell(\Ldown_r(K;\eta)) = \Phi_\ell(\Ldown_1(\lk \eta))= \Phi_\ell(\Lup_0(\lk \eta)).$$
Applying Theorem \ref{brou} to the graph $\lk \eta$ gives
\begin{equation}\label{eq-brouwer}
\Phi_\ell(\Ldown_r(K;\eta))=\Phi_\ell(\Lup_0(\lk \eta))
\leq
|S_1(\lk \eta)|
+
\binom{\ell + 1}{2}.
\end{equation}

For each $\eta \in S_{r - 2}(K)$, the edges of $\lk \eta$ are in bijection with the $r$-dimensional faces of $K$ containing $\eta$.
Therefore, double counting the pairs $(\eta,\sigma)$ satisfying $\eta \in S_{r - 2}(K)$, $\sigma \in S_r(K)$, and $\eta \subset \sigma$ gives
\begin{equation}\label{eq-sum-edge}
\sum_{\eta \in S_{r - 2}(K)}
|S_1(\lk \eta)|
=
\binom{r + 1}{2} f_r(K).
\end{equation}
Indeed, every $r$-dimensional face contains exactly $\binom{r + 1}{2}$ faces of dimension $r - 2$.

Using Lemma \ref{com-graph}, \eqref{eq-brouwer}, and \eqref{eq-sum-edge}, we obtain
\[
\begin{aligned}
\sum_{i = 1}^{\ell} \lambda_{r-1,i}(K)
&\leq
\frac{1}{r}
\sum_{\eta \in S_{r - 2}(K)}
\Phi_\ell(\Ldown_r(K;\eta)) \\
&\leq
\frac{1}{r}
\sum_{\eta \in S_{r - 2}(K)}
\left(
|S_1(\lk \eta)|
+
\binom{\ell + 1}{2}
\right) \\
&=
\frac{1}{r}
\binom{r + 1}{2} f_r(K)
+
\frac{f_{r - 2}(K)}{r}
\binom{\ell + 1}{2} \\
&=
\frac{r + 1}{2} f_r(K)
+
\frac{f_{r - 2}(K)}{r}
\binom{\ell + 1}{2}.
\end{aligned}
\]
\end{proof}

\section{Counterexample to the Duval-Reiner conjecture}\label{counter}
We begin with a pure two-dimensional complex that serves as the base of the construction.
Let $K$ be the pure two-dimensional simplicial complex with vertices  $1,\ldots,7$ whose facets, in the displayed order, are
\[
\begin{aligned}
S_2(K) = \{&124,126,127,134,135,136,137,145,146,147,\\
&156,157,167,256,456,457,467,567\}.
\end{aligned}
\]
Here, $abc$ denotes the face $\{a,b,c\}$.
The complex has $20$ edges, and its vertex degrees are
\[
(\deg_0(K;1),\deg_0(K;2),\ldots,\deg_0(K;7)) = (13,4,4,8,8,9,8).
\]
Therefore,
\[
\begin{aligned}
\dd_0^\top(K) &= (7,7,7,7,5,5,5,5,2,1,1,1,1).
\end{aligned}
\]

Orient every face by the natural increasing order of its vertices.
An exact computation gives the characteristic polynomial of $\Lup_1(K)$: 
\[
\det\bigl(xI_{20}-\Lup_1(K)\bigr)
=
x^6(x-7)^3(x-5)^2(x-1)^4p(x),
\]
where
\[
p(x) = x^5 - 19x^4 + 133x^3 - 413x^2 + 527x - 175.
\]
The polynomial $p(x)$ has five real roots $\alpha_1 > \cdots > \alpha_5$.
A direct computation gives
\[
6.16382 < \alpha_1 < 6.16383,
\qquad
5.83636 < \alpha_2 < 5.83637.
\]
In particular, $\alpha_1 + \alpha_2 > 12.00018 > 12$, and the first five largest eigenvalues of $\Lup_1(K)$ are $7,7,7,\alpha_1,\alpha_2$.
Consequently,
\[
\sum_{i = 1}^5 \lambda_{1,i}(K)
=
21 + \alpha_1 + \alpha_2
>
33
=
\sum_{j = 1}^5 \deg^\top_{0,j}(K).
\]
Thus, $K$ violates the fifth partial-sum inequality in Conjecture \ref{dr}.

For an integer $q \geq 0$, introduce new vertices $v_1,\ldots,v_q$ and define the pure two-dimensional complex $K_q$ by
\[
S_2(K_q)
=
S_2(K)
\cup
\{14v_h : 1 \leq h \leq q\}.
\]
Thus, $K_q$ is obtained from $K$ by attaching $q$ triangles along the edge $14$.

\begin{lem}\label{complex-sum}
For every integer $q \geq 0$, the first five largest eigenvalues of $\Lup_1(K_q)$ are
\[
q+7,\ 7,\ 7,\ \alpha_1,\ \alpha_2.
\]
\end{lem}

\begin{proof}
The result is immediate when $q = 0$, so assume that $q \geq 1$.
Order the vertices by
\[
1<2<\cdots<7<v_1<\cdots<v_q
\]
and order the new facets $14v_1,\ldots,14v_q$ after the facets of $K$.
For each $j$, taking the inner products of $\p_2([1,4,v_j])$ with $\p_2 [F]$ for each $F \in S_2(K)$ in the displayed order, gives the same vector
\[
c=(-1,0,0,-1,0,0,0,1,1,1,0,0,0,0,0,0,0,0)^\top.
\]
It follows that
\[
\Ldown_2(K_q)
=
\begin{pmatrix}
\Ldown_2(K) & c\mathbf{1}_q^\top \\
\mathbf{1}_q c^\top & 2I_q + J_q
\end{pmatrix},
\]
where $\mathbf{1}_q$ is the all-one vector of dimension $q$, and $J_q = \mathbf{1}_q\mathbf{1}_q^\top$. 
For any $\sigma \in S_2(K)$, we have
\[
\bigl(\Ldown_2(K)c\bigr)_\sigma
=
\begin{cases}
-7, & \text{if ~ } \sigma \in \{124,134\},\\
~ 7, & \text{if ~ }\sigma \in \{145,146,147\},\\
~ 0, & \text{otherwise}.
\end{cases}
\]
Therefore, $\Ldown_2(K)c = 7c$, and  $c^\top c = 5$.

The subspace
\[
W = \{\mathbf{0}_7\} \oplus \{\xi \in \R^q : \mathbf{1}_q^T\xi = 0\}
\]
is a $(q - 1)$-dimensional eigenspace of $\Ldown_2(K_q)$ associated with the eigenvalue $2$, where $\mathbf{0}_t$ denotes the zero vector of dimension $t$.

The subspace
\[
U = \operatorname{span}\{(c,\mathbf{0}_q),(\mathbf{0}_7,\mathbf{1}_q)\}
\]
is invariant under $\Ldown_2(K_q)$, and the restriction of $\Ldown_2(K_q)$ to $U$ is represented, in the displayed ordered basis, by the matrix
\[
\begin{pmatrix}
7 & q \\
5 & q + 2
\end{pmatrix}.
\]
Its eigenvalues are $2$ and $q+7$.

Since $\Ldown_2(K)$ is symmetric and $c$ is an eigenvector of $\Ldown_2(K)$ associated with the eigenvalue $7$, the subspace $c^\perp\oplus\{\mathbf{0_q}\}$ is invariant under $\Ldown_2(K_q)$, and the restriction agrees with the restriction of $\Ldown_2(K)$ to $c^\perp$.

So, we have the following decomposition:
\[
\R^{18+q}
=
(c^\perp\oplus\{\mathbf{0}_q\})\oplus U\oplus W,
\]
where $\Ldown_2(K_q)$ has eigenvalues $\lambda \in \bflam_{2}(K) \setminus \{7\}$ on $c^\perp\oplus\{\mathbf{0}_q\}$, eigenvalues $2$ and $q+7$ on $U$, and eigenvalue $2$ with multiplicity $q-1$ on $W$.
Hence, the spectrum of $\Ldown_2(K_q)$ is obtained from $\Ldown_2(K)$ by replacing one copy of $7$ by $q+7$ and adding $q$ copies of $2$.

Since $\Lup_1(K_q)$ and $\Ldown_2(K_q)$ share the same nonzero eigenvalues, the same statement holds for the positive eigenvalues of $\Lup_1(K_q)$.
Finally, $q+7>7>\alpha_1>\alpha_2>5>2$, which proves the result.
\end{proof}

We now lift the construction to arbitrary dimension.
Fix integers $r\geq 2$ and $n\geq r+5$, and set $q=n-r-5$.
Let $A$ be a set of $r-2$ new vertices disjoint from $V(K_q)$; when $r=2$, take $A=\emptyset$.
Define $K_{n,r}$ to be the pure $r$-dimensional simplicial complex whose facets are
\[
S_r(K_{n,r})
=
\{A\cup F:F\in S_2(K_q)\}.
\]
Every facet has cardinality $(r-2)+3=r+1$, and
\[
|V(K_{n,r})|=(r-2)+(7+q)=n.
\]
This construction is the join of $K_q$ with the simplex on the vertex set $A$.
\begin{lem}\label{complex-core}
Under the bijection $F\mapsto A\cup F$ from $S_2(K_q)$ to $S_r(K_{n,r})$,
\[
\Ldown_r(K_{n,r})
=
\Ldown_2(K_q)+(r-2)I_{18+q}.
\]
\end{lem}
\begin{proof}
Order the vertices of $A$ before those of $K_q$ and use the induced orientations.
For every $F\in S_2(K_q)$, the diagonal entries of $\Ldown_2(K_q)$ and $\Ldown_r(K_{n,r})$ indexed by $F$ and $A\cup F$ are $3$ and $r+1$, respectively.
Thus, each diagonal entry increases by $r-2$.

Let $F,G\in S_2(K_q)$ be distinct.
If $F$ and $G$ do not share an edge, then $A\cup F$ and $A\cup G$ do not share an $(r-1)$-face, so the corresponding off-diagonal entries of both down-Laplacians are zero.
If $e=F\cap G\in S_1(K_q)$, then $A\cup e$ is their unique common $(r-1)$-face.
Because all vertices of $A$ precede those of $K_q$,
\[
([A\cup F]:[A\cup e])=(-1)^{r-2}([F]:[e]),
\quad
([A\cup G]:[A\cup e])=(-1)^{r-2}([G]:[e]).
\]
The two factors $(-1)^{r-2}$ cancel in the product, so the corresponding off-diagonal entries are equal.
\end{proof}

\begin{prop}\label{general-counterexamples}
For every $r \geq 2$ and $n \geq r + 5$, the complex $K_{n,r}$ satisfies
\[
\sum_{i = 1}^5 \lambda_{r-1,i}(K_{n,r})
>
\sum_{j = 1}^5 \deg^\top_{0,j}(K_{n,r}).
\]
Consequently, the Duval-Reiner conjecture fails for $K_{n,r}$.
\end{prop}

\begin{proof}
By Lemmas \ref{complex-sum}, \ref{complex-core}, and the equality of the positive eigenvalues of $\Lup_{r-1}(K_{n,r})$ and $\Ldown_r(K_{n,r})$, the first five largest eigenvalues of $\Lup_{r-1}(K_{n,r})$ are
$$n, r + 5, r + 5, r - 2 + \alpha_1,
r - 2 + \alpha_2.$$

 Next, we determine the first five entries of $\dd_0^\top(K_{n,r})$.
Every vertex of $A$ belongs to all $18+q$ facets.
The degrees of the remaining vertices are
\[
\begin{aligned}
d_0(1)&=13+q,&
d_0(2)&=4,&
d_0(3)&=4,&
d_0(4)&=8+q,\\
d_0(5)&=8,&
d_0(6)&=9,&
d_0(7)&=8,&
d_0(v_h)&=1, h \ge 7.
\end{aligned}
\]
Every vertex has a positive degree, so $\deg^\top_{0,1}(K_{n,r})=n$.
For $j\in\{2,3,4\}$, precisely the $q$ vertices $v_1,\ldots,v_q$ have a degree less than $j$, and hence
\[
\deg^\top_{0,j}(K_{n,r})=n-q=r+5.
\]
For $j=5$, the vertices $v_1,\ldots,v_q,2,3$ have a degree less than $5$, and therefore
\[
\deg^\top_{0,5}(K_{n,r})=n-q-2=r+3.
\]
Thus
\[
\bigl(\deg^\top_{0,1}(K_{n,r}),\ldots,\deg^\top_{0,5}(K_{n,r})\bigr)
=
(n,r+5,r+5,r+5,r+3).
\]
It follows that
\[
\sum_{i=1}^{5}\lambda_{r-1,i}(K_{n,r})
-
\sum_{j=1}^{5}\deg^\top_{0,j}(K_{n,r})
=
\alpha_1+\alpha_2-12
>
0.
\]
Hence $K_{n,r}$ violates the fifth partial-sum inequality in Conjecture \ref{dr}.
\end{proof}

\end{document}